\documentclass[12pt]{amsart}
\usepackage{amsthm}
\usepackage{amssymb}

\usepackage[margin=1in]{geometry}

\allowdisplaybreaks
\newtheorem{theorem}{Theorem}
\newtheorem{lem}{Lemma}

\newtheorem{corollary}{Corollary}
\newtheorem{thm}{Theorem}
\theoremstyle{definition}

\theoremstyle{remark}
\newtheorem{rem}{Remark}
\newtheorem{remark}{Remark}

\DeclareMathOperator{\mt}{mt}
\usepackage[normalem]{ulem}
\usepackage{soul}

\usepackage{cite}

\usepackage{etoolbox}
\makeatletter
\patchcmd{\@settitle}{\uppercasenonmath\@title}{}{}{}
\patchcmd{\@setauthors}{\MakeUppercase}{}{}{}
\patchcmd{\section}{\scshape}{\bfseries}{}{}
\makeatother

\begin{document}
\title{Sampling in de Branges Spaces and Naimark Dilation}%
\author{Sa'ud al-Sa'di}
\author{Eric S. Weber}
\address{Department of Mathematics, Hashemite University, Zarqa, Jordan}
\email{saud@hu.edu.jo}
\address{Department of Mathematics, Iowa State University, 396 Carver Hall, Ames, IA 50011}
\email{esweber@iastate.edu}

\date{\today}


%



\begin{abstract}

We consider the problem of sampling in de Branges spaces and develop some necessary conditions and some sufficient conditions for sampling sequences, which generalize some well-known sampling results in the Paley-Wiener space. These conditions are obtained by identifying the main construction with Naimark dilation of frames--embedding the de Branges space into a larger de Branges space while embedding the kernel functions associated with a sampling sequence into a Riesz basis  for the larger space.
\end{abstract}
\maketitle

\section{Introduction}

The problem of sampling in the Paley-Wiener space of bandlimited functions began with the Shannon-Whitaker-Kotelnikov theorem (see \cite{BF01a} for a history), which says that if a function $f$ is in $L^{2}(\mathbb{R})$ and band limited to $(-\pi,\pi)$, then it can be recovered from its samples $\{ f(n) \}_{n}$ on the integer lattice via cardinal interpolation.  Later, Duffin and Schaeffer \cite{DS52a}, Landau \cite{L67a}, and others gave necessary and (different) sufficient conditions for a sequence $\{ \lambda_{n} \}_{n}$ such that $f$ can be recovered from its samples $\{ f(\lambda_{n}) \}_{n}$.  A complete description of the sampling sequences for the Paley-Wiener space was given by Ortega-Cerda and Seip \cite{OS02a}.

Using the theory of de Branges spaces, Ortega-Cerda and Seip characterize the sampling sequences in essentially the following way (please see Theorem \ref{thm:ocs} for the precise statement):  the sequence is a sampling sequence if and only if the Paley-Wiener space can be embedded into a larger space in such a way that the sequence becomes a complete interpolating sequence.  

The purpose of the present paper is to demonstrate that when this description is viewed from the perspective of frame theory, it says that the kernel functions associated to $\{ \lambda_{n} \}_{n}$ form a frame if and only if the Paley-Wiener space can be embedded into a larger space in such a way that the kernel functions can be ``dilated'' to a Riesz basis for the larger space.  This process of embedding a frame into a Riesz basis is referred to as Naimark dilation \cite{HL00a}.

We will demonstrate that the characterization of the sampling sequences given by Ortega-Cerda and Seip does in fact correspond to Naimark dilation.  We will also show that for other de Branges spaces, the same Naimark dilation phenomenon holds to describe the sampling sequences in those spaces.

\subsection{Frame Theory}

A sequence $\{ f_{n};n\in I\}$ is a frame for a separable Hilbert space $\mathcal{H}$ if there exists constants $0<A\leq
B<\infty$ such that
\begin{equation}\label{eqn:framecondition}
 A \|f\|^{2}\leq\sum_{n\in I}|\langle f,f_{n}\rangle|^{2}\leq B
\|f\|^{2}, \;\; \text{ for all}\;\; f \in \mathcal{H},
\end{equation}

The constants $A$ and $B$ are called lower and upper frame bounds, respectively. The frames for which $A=B=1$ are called \emph{Parseval frames}. A frame which is a basis is called a \emph{Riesz basis}. It is easy to see that a Parseval frame $\{ f_{n};n\in I\}$ for a Hilbert space $\mathcal{H}$ is an orthonormal basis if and only if each $f_{n}$ is a unit vector. If  $\{ f_{n};n\in I\}$ satisfies the second inequality, then $\{ f_{n};n\in I\}$ is called a \emph{Bessel} sequence.

Let $\{f_{n}\}_{n\in I}$ be a Bessel sequence in $\mathcal{H}$. The analysis operator $\Theta: \mathcal{H}\to \ell^{2}(I)$, which is bounded because of (\ref{eqn:framecondition}), is defined by
\[ \Theta: f\to (\left< f,f_{n}\right>);\]
and the synthesis operator $\Theta^{*}:\ell^{2}(I) \to \mathcal{H}$, which is  the adjoint operator of $\Theta$, is defined by
\[\Theta^{*}: (c_{n})_{n\in I}\to \sum_{n\in I}c_{n}f_{n}.\]
Additionally, the sum
$\sum_{n\in I}c_{n}f_{n}$ converges in $\mathcal{H}$ for all $(c_{n})_{n\in I}\in \emph{l}^{2}(I)$ (see \cite{DS52a}), and so the synthesis operator is also well defined and bounded.

 The operator $S:=\Theta^{*}\Theta:\mathcal{H}\to \mathcal{H}$ is called the frame operator, and we have
     \[ Sf=\sum_{n\in I}\langle f,f_{n} \rangle f_{n}\;, \forall f\in \mathcal{H}.\]

 The \emph{canonical dual frame} is denoted by $\{\tilde{f}_{n}\}_{n\in I}$, and is defined by $\tilde{f}_{n}=S^{-1}f_{n}$. Furthermore, for each $f\in \mathcal{H}$ we have the \emph{frame expansions}
\begin{equation}
f=\sum_{n\in I}\langle f,f_{n} \rangle \tilde{f}_{n}= \sum_{n\in I}\langle f,\tilde{f}_{n} \rangle f_{n},
\end{equation}
with unconditional convergence of these series.

If $\mathbb{F}=\{f_{n}\}_{n\in I}$ and $\mathbb{G}=\{g_{n}\}_{n\in I}$ are two Bessel sequences in $\mathcal{H}$, define the operator \[ \Theta^{*}_{\mathbb{G}}\Theta_{\mathbb{F}}:\mathcal{H} \to \mathcal{H}: f\to \sum_{n\in I}\left< f,f_{n}\right>g_{n}.\]

If $\Theta^{*}_{\mathbb{G}}\Theta_{\mathbb{F}}=0$ then the two Bessel sequences $\mathbb{F}$ and $\mathbb{G}$ are said to be orthogonal \cite{HKLW07a}. An extensive study of orthogonal frames can be found in the papers \cite{BJW07a,W04a}. If $\mathbb{F}$ and $\mathbb{G}$ are both Parseval frames and orthogonal to each other, then for any $f,g\in \mathcal{H}$
\[  f= \sum_{n}(\left< f,f_{n}\right>+\left< g,g_{n}\right>)f_{n}, \;\; \text{and} \;\; g= \sum_{n}(\left< f,f_{n}\right>+\left< g,g_{n}\right>)g_{n}\]
In other words,  both functions can be recovered from the summed coefficients $\left< f,f_{n}\right>+\left< g,g_{n}\right>$. This procedure is called \emph{multiplexing}, and can be used in multiple access communication systems. In the proof of our main results we also need a concept of \emph{similar frames}: two frames $\mathbb{F}=\{f_{n}\}_{n\in I}$ and $\mathbb{G}=\{g_{n}\}_{n\in I}$ are said to be similar if there is an invertible operator $T:\mathcal{H} \to \mathcal{H}$ such that $Tf_{n}=g_{n}$.  Two frames $\mathbb{F}$ and $\mathbb{G}$ are similar if and only if $\Theta_{\mathbb{F}}(H) = \Theta_{\mathbb{G}}(H)$ \cite{Cas00a}.

Let $P$ be an orthogonal projection from a Hilbert space $\mathcal{K}$ onto a closed subspace $\mathcal{H}$, and $\{f_{n}\}$ be a sequence in $\mathcal{K}$. Then $\{Pf_{n}\}$ is called \emph{orthogonal compression} of $\{f_{n}\}$ under $P$, and $\{f_{n}\}$ is called an \emph{orthogonal dilation} of $\{Pf_{n}\}$.  A classical fact on dilation of frames,  which can be attributed to Han and Larson \cite{HL00a},  says that a Parseval frame in a Hilbert space $\mathcal{H}$ is an image of an orthonormal basis under an orthogonal projection of some larger Hilbert space $\mathcal{K}\supseteq  \mathcal{H}$ onto $\mathcal{H}$. This result  can be considered as a special case of Naimark's dilation theorem for positive operator valued measures, see \cite{Naimark40,Naimark43}. In particular, Han and Larson proved the following result.

\begin{thm} \label{thm:Han&Larson-Theorem}
Let $\{f_{n}\}_{n\in I}$ be a sequence in a Hilbert space $\mathcal{H}$. Then
\begin{itemize}
\item [($i$)] $\{f_{n}\}$ is a Parseval frame for $\mathcal{H}$ if and only if there exists a Hilbert space $\mathcal{K} \supseteq \mathcal{H}$ and an orthonormal basis $\{e_{n}\}$ for $\mathcal{K}$ such that if $P$ is the orthogonal projection of $\mathcal{K}$ onto $\mathcal{H}$ then $f_{n}=Pe_{n}$, for all $n\in I$.

\item [($ii$)] $\{f_{n}\}$ is a frame for $\mathcal{H}$ if and only if there exists a Hilbert space $\mathcal{K} \supseteq \mathcal{H}$ and a Riesz basis $\{u_{n}\}$ for $\mathcal{K}$ such that if $P$ is the orthogonal projection of $\mathcal{K}$ onto $\mathcal{H}$ then $f_{n}=Pu_{n}$, for all $n\in I$.

\end{itemize}
\end{thm}

Orthogonality of frames and Naimark dilation of frames are related in the following way \cite{HL00a,B99a}.  If $\{u_n\}$ is a Riesz basis for $\mathcal{K}$ and $P$ is the projection onto $\mathcal{H} \subset \mathcal{K}$, then $\{ P u_{n} \}$ and $\{ (I-P) u_{n} \}$ are orthogonal frames for $\mathcal{H}$ and $\mathcal{H}^{\perp}$, respectively.  Conversely, if $\mathbb{F} = \{ f_{n} \}$ and $\mathbb{G} = \{ g_{n} \}$ are orthogonal frames for $\mathcal{H}_{1}$ and $\mathcal{H}_{2}$, respectively, then $\{ f_{n} + g_{n} \}$ is a frame for $\mathcal{H}_{1} \oplus \mathcal{H}_{2}$.  Note that the sum of the frames need not be a basis for the direct sum in general--however, it will be provided 
\[ \Theta_{\mathbb{F}}(\mathcal{H}_{1}) \oplus \Theta_{\mathbb{G}}(\mathcal{H}_{2}) = \ell^2(I). \]

\subsection{de Branges Spaces}

We are interested in Hilbert spaces of entire functions as first introduced by L. de Branges in the series of papers \cite{dB59, dB60a, dB61a, dB61b}. These spaces, which are now called de Branges spaces, generalize the classical Paley-Wiener space which consists of the entire functions which are of exponential type $\pi$ and square integrable on the real line.

An entire function $E(z)$ is said to be of\emph{ Hermite$-$Biehler} class, denoted by $\mathcal{HB}$, if
it satisfies the condition
\begin{equation}\label{eq:deBranges-Function-definition}
|E(\bar{z})| <|E(z)|,
\end{equation}
for all $z\in \mathbb{C}^{+}=\{z\in \mathbb{C}\,:\, \text{Im} \,z>0\}$. An analytic function $f$ on  $\mathbb{C}^{+}$ is said to be of \emph{bounded type} in $\mathbb{C}^{+}$ if it can be represented as a quotient of two bounded analytic functions in $\mathbb{C}^{+}$.  The \emph{mean type} of $f$ in $\mathbb{C}^{+}$ is defined by
\begin{equation}\label{eqn:Mean Type formula}
 \mt_{+}(f)\,:=\limsup_{y\rightarrow +\infty}\frac{\log|f(iy)|}{y}
 \end{equation}

For an entire function $f$ define the function $f^{*}$ as $f^{*}(z):=\overline{f(\bar{z})}$. Given a function $E\in \mathcal{HB}$, the de Branges space $\mathcal{H}(E)$ consists of all entire functions $f(z)$ such that
\begin{equation}\label{eq:NormofDeBrangesSpace}
  ||f||^{2}_{_{E}}:=\int_{\mathbb{R}}\bigg|\frac{f(t)}{E(t)}\bigg|^{2}dt < \infty,
\end{equation}
  and
  $f(z)/E(z)$ and $f^{*}(z)/E(z)$ are of bounded type and nonpositive mean type in the upper half-plane. It is a Hilbert space with inner product defined by
\[ \langle f,g \rangle_{E}=\int_{\mathbb{R}}\frac{f(t)\overline{g(t)}}{|E(t)|^{2}}\,dt.\]
If $E,F \in \mathcal{HB}$, we write $H(E) = H(F)$ if they coincide as sets and the norms are equivalent.

\begin{thm} \label{characterizationOfH(E)inH1H2H3}
The space $\mathcal{H}(E)$ satisfies the following properties:
\begin{enumerate}
\item [(H1)] If $f \in \mathcal{H}(E)$ and $w\in \mathbb{C} \setminus \mathbb{R}$ with $f(w)=0$, then the function
 $g(z)=f(z)\frac{z-\bar{w}}{z-w}$ belongs to $ \mathcal{H}(E)$, and $\|g\|=\|f\|$.
\item [(H2)] For every nonreal number $w$, the linear functional defined on the space by $f \mapsto f(w)$ is continuous. 
\item [(H3)]  If  $f \in \mathcal{H}(E)$, then $ f^{*} \in \mathcal{H}$ and $ \|\,f^{*}\|=\|\,f\|$.
\end{enumerate}
\end{thm}

Conversely every de Branges space can be obtained in this way, see \cite{dB68a}:
\begin{thm}\label{thm:H1-H3 de Branges theorem}
A Hilbert space $\mathcal{H}$ whose elements are entire functions, which satisfies $(H1)$, $(H2)$, and $(H3)$, and which contains a nonzero element,
is equal isometrically to some space $\mathcal{H}(E)$.
\end{thm}


By (H2), for every nonreal $w\in \mathbb{C}$ there exists a reproducing kernel for $\mathcal{H}(E)$, which is given by
\begin{equation}\label{eq:reproducingkerneldefn}
K_{E}(w,z)=\frac{\bar{E}(w)E(z)-E(\bar{w})E^{*}(z)}{2\pi i(\bar{w}-z)},
\end{equation}
whence
\begin{equation}\label{eq:reproducingkernelProperty}
f(w)=\langle f(t), K_{E}(w,t)\rangle_{E}, 
\end{equation}
for every $f\in \mathcal{H}(E)$.  The kernel $K_{E}$ can be extended so that both Equations  (\ref{eq:reproducingkerneldefn}) and (\ref{eq:reproducingkernelProperty}) are satisfied for real $w$ as well.

An important feature of the de Branges space $\mathcal{H}(E)$ is the phase function corresponding to the generating function $E$. For any entire function $E\in \mathcal{HB}$, there exists a continuous and strictly increasing function  $\varphi:\mathbb{R} \rightarrow \mathbb{R}$ such that $E(x)e^{i\varphi(x)}\in \mathbb{R}$ for all $x\in \mathbb{R}$, and $E(x)$ can be written as
\begin{equation}\label{eqn:phase function definition}
E(x)=|E(x)|e^{-i\varphi(x)}, \;\; x\in \mathbb{R}
\end{equation}

If a function $\varphi$ has these properties then it is referred to as a \emph{phase function} of $E$. It follows that a phase function of $E$ is defined uniquely up to an additive constant, a multiple of $2\pi$. If $\varphi(x)$ is any such function, and $E(x)\neq 0$, then using (\ref{eq:reproducingkerneldefn}) and (\ref{eqn:phase function definition}), an easy computation gives
\begin{equation}\label{eqn:K_{E}(x,x) as phase function derivative}
\|K_{E}(x,.)\|^{2}=K_{E}(x,x)=\frac{1}{\pi}\varphi'(x)|E(x)|^{2}.
\end{equation}

A key feature of a de Branges space is that it always has a basis consisting of reproducing kernels corresponding to real points, \cite{dB68a}.
\begin{thm}\label{thm:Orthonormal Basis in H(E) Theorem}
 Let $\mathcal{H}(E)$ be a de Branges space and  $\varphi(x)$ be a phase function associated with $E$. If $\alpha\in \mathbb{R}$, and   $\Lambda=\{\lambda_{n}\}_{n\in \mathbb{Z}}$ is a sequence of real numbers, such that $\varphi(\lambda_{n})=\alpha+\pi n$, $n\in\mathbb{Z}$, then
\begin{enumerate}
\item  The functions $\{K_{E}(\lambda_{n},z)\}_{n\in \mathbb{Z}}$ form an orthogonal set in $\mathcal{H}(E)$.

\item If $e^{i\alpha}E(z)-e^{-i\alpha}E^{*}(z) \notin \mathcal{H}(E)$, then
$\big\{\frac{K_{E}(\lambda_{n},z)}{\|K_{E}(\lambda_{n},.)\|}\big\}_{n\in \mathbb{Z}}$ is an orthonormal basis for $\mathcal{H}(E)$. Moreover, for every $f(z)\in \mathcal{H}(E)$,
\begin{equation}\label{eq:SamplingFormulawithONB}
 f(z)=\sum_{n\in\mathbb{Z}}f(\lambda_{n})\frac{K_{E}(\lambda_{n},z)}{\|K_{E}(\lambda_{n},.)\|^{2}}\, ,
\end{equation}
and
\begin{equation}\label{eq:norm using ONB}
 \|f\|^{2}=\sum_{n\in\mathbb{Z}}\bigg|\frac{f(\lambda_{n})}{E(\lambda_{n})}\bigg|^{2}\frac{\pi}{\varphi'(\lambda_{n})}.
\end{equation}


\end{enumerate}
\end{thm}

By a Lemma of \cite{dB59}, there is at most one real number $\alpha$ modulo $\pi$ such
that the function $e^{i\alpha}E(z)-e^{-i\alpha}E^{*}(z)$ belongs to $ \mathcal{H}(E)$.

Given a de Branges space $\mathcal{H}(E)$, the function $E$ can be factored out into a product of two entire functions:  $E=S \tilde{E}$,  where $\tilde{E}\in \mathcal{HB}$ and has no real zeros, and $S(z)$, is real for real $z$, and has only real zeros, and $\mathcal{H}(E) = S \cdot \mathcal{H}(\tilde{E})$. Hence, we can assume without loss of generality that the function $E$ has no zeros on the real axis, see Lemma 4.1.10 of \cite{aSW14a}. 

\section{Sampling Sequences in de Branges Spaces: Necessary Conditions}

We say a sequence $\Lambda=\{\lambda_{n}\}_{n\in\mathbb{Z}}$ is \emph{separated} (or $\delta$-uniformly separated) if there exists $\delta>0$, such that $\inf\limits_{n\neq m}|\lambda_{n} - \lambda_{m}|\geq \delta>0$. The constant $\delta$ is called the separation constant of $\Lambda$.

We say $\{ \lambda_{n} \}_{n} \subset \mathbb{R}$ is a sampling sequence for $\mathcal{H}(E)$ if there exist constants $A,B > 0$ such that for all $f \in \mathcal{H}(E)$
\begin{equation} \label{Eq:sampling1}
A \| f \|_{E}^{2} \leq \sum_{n} | f(\lambda_{n}) |^2 \leq B \| f \|_{E}^{2}.
\end{equation}
We say $\{ \lambda_{n} \}_{n} \subset \mathbb{R}$ is a \emph{normalized} sampling sequence for $\mathcal{H}(E)$ if there exist constants $\widetilde{A}, \widetilde{B} > 0$ such that for all $f \in \mathcal{H}(E)$
\begin{equation} \label{Eq:sampling2}
\widetilde{A} \| f \|_{E}^{2} \leq \sum_{n} \dfrac{| f(\lambda_{n}) |^2}{K_{E}(\lambda_{n}, \lambda_{n})} \leq \widetilde{B} \| f \|_{E}^{2}.
\end{equation}
Note that if $K_{E}(x,x) \simeq 1$, then the inequalities (\ref{Eq:sampling1}) and (\ref{Eq:sampling2}) are equivalent, as happens in the Paley-Wiener space.  However, the two inequalities are not equivalent in general.  We say that the sequence $\{ \lambda_{n} \}_{n} \subset \mathbb{R}$ is a Plancherel-Polya sequence, respectively normalized Plancherel-Polya sequence, if it satisfies (possibly only) the upper inequality of (\ref{Eq:sampling1}) or (\ref{Eq:sampling2}).

  The leading example of a de Branges space is the Paley-Wiener space $PW_{a}$, $a>0$, the space of entire functions which are square integrable on the real line and are of exponential type $a$.  In this case we can write $PW_{a}=\mathcal{H}(E)$, where $E(z)=\exp(-iaz)$.  Landau proved necessary density conditions for sampling sequences in the Paley-Wiener space \cite{L67a}. Landau's results were reproven by Gr\"{o}chenig and Razafinjatovo \cite{GR96a} using an argument based on the Homogeneous Approximation Property. Lyubarskii and Seip \cite{LS02a} extend Landau's necessary density criteria to de Branges spaces where the phase function satisfies  the condition  $\varphi'(x) \simeq 1$. Marzo, Nitzan, and Olsen \cite{MNO12} extend Landau's results to de Branges spaces which have the property that the measure $\varphi'(x)dx$ is a ``doubling measure".

  A complete characterization of which sequences are sampling in the Paley-Wiener $PW_{\pi}$ was obtained by Ortega-Cerd\'{a} and Seip \cite{OS02a}:

  \begin{thm} \label{thm:ocs}
  A separated sequence $\Lambda$ of real numbers is sampling for $PW_{\pi}$ if and only if there exist two entire functions $E, F \in \mathcal{HB}$ such that
  \begin{itemize}
  \item[($i$)] $H(E)=PW_{\pi}$
  \item[($ii$)] $\Lambda$ constitutes the zero sequence of $EF+E^{*}F^{*}$.
  \end{itemize}
  \end{thm}

The reproducing kernel property (\ref{eq:reproducingkernelProperty}) implies that a sequence $\Lambda=\{\lambda_{n}\}_{n\in\mathbb{Z}}$ is a sampling sequence in $\mathcal{H}(E)$ if and only if the corresponding sequence of reproducing kernels $\{K_{E}(\lambda_{n}, \cdot)\}_{n\in \mathbb{Z}}$ is a \emph{frame} for $\mathcal{H}(E)$, therefore, any function $f\in \mathcal{H}(E)$ can be reconstructed from its samples on the sequence $\Lambda$ by the (sampling) formula
\begin{equation}
  f(z)=\sum_{n\in \mathbb{Z}} f(\lambda_{n})\, \tilde{k}_{n}(z)
\end{equation}
 where $\{\tilde{k}_{n}\}_{n\in \mathbb{Z}}$  is a dual frame of $\{K(\lambda_{n}, \cdot)\}_{n\in \mathbb{Z}}$.  The same can be said of a normalized sampling sequence and normalized kernels $\left\{ \dfrac{K_{E}(\lambda_{n}, \cdot)}{ \| K_{E}(\lambda_{n}, \cdot ) \|_{E}} \right\}_{n}$.

Characterizing sampling sequences in de Branges spaces other than the Paley-Wiener spaces is unresolved in general. Our first main result is an extension of Ortega-Cerd\'{a} and Seip's necessary conditions for a sequence to be sampling for a de Branges space. The argument proceeds almost identically to the one given in \cite{OS02a}.

\begin{lem}\label{lem:condition on A(z)}
Suppose that $E \in \mathcal{HB}$, and suppose that $\mu = \sum_{n} | E(\lambda_{n}) |^2 \delta_{\lambda_{n}}$ satisfies
\[ \| f \|_{E}^2 = \int  \left| \dfrac{f(t)}{E(t)} \right|^2 d \mu(t). \]
Then there exists a function $A \in H^{\infty}(\mathbb{C}^{+})$ with $\| A \|_{\infty} \leq 1$ such that
\begin{equation}
\frac{i}{\pi} \sum_{n} | E (\lambda_{n} ) |^2 \left( \dfrac{ 1 }{ z - \lambda_{n} } + \dfrac{ \lambda_{n} }{ 1 + \lambda_{n}^{2} } \right) + i a = \dfrac{ E(z)+ E^{*}(z) A(z) }{ E(z) - E^{*}(z) A(z)}
\end{equation}
for some $a \in \mathbb{R}$ and all $z \in \mathbb{C}^{+}$.
\end{lem}

\begin{proof}
There exists an $A \in H^{\infty}(\mathbb{C}^{+})$ of norm at most 1 (see \cite[Pg. 90]{dB68a}) such that
\[ Re \left( \dfrac{ E(z)+ E^{*}(z) A(z) }{ E(z) - E^{*}(z) A(z)} \right) = \dfrac{y}{\pi} \int \dfrac{ d \mu(t)}{ (x - t)^{2} + y^2 } =: V(z). \]
A harmonic conjugate for $V(z)$ is given by (see \cite{Koosis98}, pg 109):
\[ \widetilde{V}(z) = \dfrac{1}{\pi} \int \dfrac{ x - t }{ (x - t)^2 + y^2 } + \dfrac{ t }{ 1 + t^2} d \mu(t). \]
Thus, we obtain for some $a \in \mathbb{R}$
\begin{align*}
\dfrac{ E(z)+ E^{*}(z) A(z) }{ E(z) - E^{*}(z) A(z)} &= V(z) + i \widetilde{V}(z) + ia \\
&= \dfrac{1}{\pi} \sum_{n} | E(\lambda_{n}) |^2 \left( \dfrac{ y + i(x - \lambda_{n} ) }{ (x - \lambda_{n})^2 + y^2 } + \dfrac{ i \lambda_{n} }{ 1 + \lambda_{n}^2} \right) + ia \\
&= \dfrac{1}{\pi} \sum_{n} | E(\lambda_{n}) |^2 \left( \dfrac{ i (\overline{z} - \lambda_{n}) }{ (z - \lambda_{n}) (\overline{z} - \lambda_{n}) } + \dfrac{ i \lambda_{n} }{ 1 + \lambda_{n}^2} \right) + ia \\
&= \dfrac{i}{\pi} \sum_{n} | E(\lambda_{n}) |^2 \left( \dfrac{ 1 }{ z - \lambda_{n} } + \dfrac{ \lambda_{n} }{ 1 + \lambda_{n}^2} \right) + ia.
\end{align*}

\end{proof}


\begin{theorem}\label{thm:main theorem of E and F}
Let $E_{0} \in \mathcal{HB}$.  If $\Lambda$ is sampling sequence for $\mathcal{H}(E_{o})$, then there exists two functions $E,F\in \mathcal{HB}$ such that
\begin{itemize}
\item [(i)] $\mathcal{H}(E_{o})\simeq \mathcal{H}(E)$,
\item [(ii)] $\Lambda$ constitutes the zero sequence of $EF+E^{*}F^{*}$.
\end{itemize}
\end{theorem}
\begin{proof}
Since $\Lambda=\{\lambda_{n}\}_{n\in \mathbb{Z}}$ is a sampling sequence for $\mathcal{H}(E_{o})$, there exists $A,B>0$ such that
\[ A \|f\|^{2}_{E_{o}} \;\leq\; \sum_{n\in \mathbb{Z}} |f(\lambda_{n})|^{2} \;\leq \;B \|f\|^{2}_{E_{o}} \]
for all $f\in\mathcal{H}(E_{o})$. This means that the space $\mathcal{H}(E_{o})$ equipped with the norm
\[\left(\sum_{n} |f(\lambda_{n})|^{2} \right)^{1/2}\]
 is a de Branges space.

Theorem \ref{thm:H1-H3 de Branges theorem} provides us with a function $E\in \mathcal{HB}$ such that $\mathcal{H}(E_{o})\simeq \mathcal{H}(E)$, and
\begin{equation}\label{eqn:equalityofNorms-of-TheoremC}
 \sum_{n} |f(\lambda_{n})|^{2}= \int_{\mathbb{R}}\bigg|\frac{f(t)}{E(t)}\bigg|^{2}dt
 \end{equation}
for all $f\in\mathcal{H}(E_{o})$.

To prove the existence of the function $F$,  Lemma (\ref{lem:condition on A(z)})  proves that there exists a bounded holomorphic function $A$ in the upper half-plane with norm $\|A\|_{\infty}\leq 1$ and a real number $a$ such that
\begin{equation}\label{eqn:definition-of-M(z)}
 \frac{i}{\pi}  \sum_{n} |E(\lambda_{n})|^{2}\bigg( \frac{1}{z-\lambda_{n}}+\frac{\lambda_{n}}{1+\lambda_{n}^{2}}  \bigg)+ia = \frac{E(z)+E^{*}(z)A(z)}{E(z)-E^{*}(z)A(z)}
\end{equation}

Note that the right-hand side is a holomorphic function defined in the upper half plane, and the left-hand side is a meromorphic function defined in the whole plane.

Let $M(z):= \frac{E(z)+E^{*}(z)A(z)}{E(z)-E^{*}(z)A(z)}$. Then by the left hand side of Equation (\ref{eqn:definition-of-M(z)}), $M^{*}=-M$, and
\[A(z)=\frac{M(z)-1}{M(z)+1}\; \frac{E(z)}{E^{*}(z)},\]
for $z \in \mathbb{C}^{+}$. From equation (\ref{eqn:definition-of-M(z)}), the function $M(z)-1$ has poles at the $\lambda_{n}$'s. Moreover, the function $M(z)-1$ vanishes whenever $E^{*}(z)$ vanishes.

Now, define
\[G(z):= \prod_{n} \bigg( 1-\frac{z}{\lambda_{n}} \bigg)e^{z/\lambda_{n}}.  \]
Note that $G^{*}=G$, and $G$ vanishes only at the $\lambda_{n}$'s. Define an entire function $F$ by
\[ F^{*}(z):= -\frac{(M(z)-1)G(z)}{E^{*}(z)}.\]
Then,
\[M(z)-1 = -\; \frac{E^{*}(z)F^{*}(z)}{G(z)},\]
and since $M^{*} = -M$, we also have,
\begin{equation}\label{eqn:M(z)+1}
     M(z)+1 = \frac{E(z)F(z)}{G(z)},
\end{equation}
which implies that $F^{*}(z)/F(z)=-A(z)$, for all $z\in \mathbb{C}^{+}$, and $F$ has no zeros in $\mathbb{C}^{+}$. Since $\|A\|_{\infty}\leq 1$, this implies that $|F^{*}(z)|<|F(z)|$ for all $z\in \mathbb{C}^{+}$ , and hence, $F\in \mathcal{HB}$.

Now we will see that $\Lambda$ is the zero set of $EF+E^{*}F^{*}$. First note that from (\ref{eqn:M(z)+1}), $G(z)=-M(z)G(z)+E(z)F(z)$, for all $z\in \mathbb{C}$. We know that if $x\in \mathbb{R}$, then $G(x)$ is real, and by the left hand side of Equation (\ref{eqn:definition-of-M(z)}), $M(x)G(x)$ is an imaginary number. Thus, $G$ is the real part of $EF$ for real $z$. In other words,
\[G(z)=\frac{E(z)F(z)+E^{*}(z)F^{*}(z)}{2},\]
for all $z\in \mathbb{C}$. This implies that $\Lambda$ is the zero set of $EF+E^{*}F^{*}$, because $G(\lambda_{n})=0$ for all $n\in \mathbb{Z}$.
\end{proof}

\subsection{Duality of Kernel Functions}

\begin{theorem} \label{thm:similarity}
Fix $E_{0} \in \mathcal{HB}$ and suppose $\{ \lambda_{n} \}_{n} \subset \mathbb{R}$ is a sampling sequence in $\mathcal{H}(E_{0})$.  Let $E$ be the Hermite-Biehler function given by Theorem \ref{thm:main theorem of E and F}.  Then the kernel functions $\{ K_{E}(\lambda_{n}, \cdot ) \}_{n}$ form a frame in $\mathcal{H}(E_{0})$, and for every $f \in \mathcal{H}(E_{0})$,
\[ f(z) = \sum_{n} f(\lambda_{n}) K_{E}(\lambda_{n}, z), \]
with convergence in the norm in $\mathcal{H}(E_{0})$.  Moreover, the frame $\{ K_{E}(\lambda_{n}, \cdot ) \}_{n}$ is the canonical dual frame for $\{ K_{E_{0}}(\lambda_{n}, \cdot ) \}_{n}$.
\end{theorem}

\begin{proof}
Recall that $E$ is defined such that $\mathcal{H}(E_{0}) = \mathcal{H}(E)$ (with equivalent norms) and for every $f \in \mathcal{H}(E_{0})$,
\[ \sum_{n} | f(\lambda_{n}) |^2 = \| f \|_{E}^{2}. \]
Thus, it follows that $\{ K_{E}(\lambda_{n} , \cdot ) \}_{n}$ is a Parseval frame in $\mathcal{H}(E)$, hence
\begin{equation} \label{Eq:reconstruct}
f(z) = \sum_{n} \langle f , K_{E}(\lambda_{n}, \cdot ) \rangle_{E} K_{E}(\lambda_{n}, z ) = \sum_{n} f(\lambda_{n}) K_{E}(\lambda_{n}, z )
\end{equation}
with convergence in $\mathcal{H}(E)$.  Since the norms are equivalent, the sum in Equation (\ref{Eq:reconstruct}) converges in $\mathcal{H}(E_{0})$.  Also as a consequence of the equivalent norms, the inclusion $I : \mathcal{H}(E) \to \mathcal{H}(E_{0}): f \mapsto f$ is an invertible operator, and so $\{ K_{E}(\lambda_{n}, \cdot ) \}$ is a frame in $\mathcal{H}(E_{0})$.  Combining the previous observations demonstrates the duality.  To establish the canonical duality, we claim that $\{ K_{E_{0}}( \lambda_{n}, \cdot ) \}_{n}$ and $\{ K_{E} ( \lambda_{n}, \cdot ) \}_{n}$ are similar as frames in $\mathcal{H}(E_{0})$.  The frame $ \{ K_{E_{0}} (\lambda_{n}, \cdot ) \}_{n} \subset \mathcal{H}(E_{0})$ is similar to $ \{ K_{E}(\lambda_{n} , \cdot ) \}_{n} \subset \mathcal{H}(E)$ since they have the same coefficient sequences, namely $ \{ ( f( \lambda_{n}))_{n} : f \in \mathcal{H}(E) \}$.  As noted before, the inclusion mapping from $\mathcal{H}(E)$ to $\mathcal{H}(E_{0})$ is a similarity, so the frame $\{ K_{E}(\lambda_{n}, \cdot )\}_{n}$ in $\mathcal{H}(E)$ is similar to itself in $\mathcal{H}(E_{0})$.  It follows that $\{ K_{E}(\lambda_{n}, \cdot )\}_{n}$ and $\{ K_{E_{0}}(\lambda_{n}, \cdot )\}_{n}$ are similar as frames in $\mathcal{H}(E_{0})$; consequently, they are canonical dual frames of each other.
\end{proof}

\section{Orthogonality in $\mathcal{H}(EF)$}

As we will show in the present section, Theorem \ref{thm:main theorem of E and F} is essentially performing Naimark dilation on the frame $\{ K_{E_0}(\lambda_{n}, \cdot ) \}_{n}$.  Theorem \ref{thm:main theorem of E and F} provides an embedding of $\mathcal{H}(E_{0})$ (and $\mathcal{H}(F)$) into a larger space, namely $\mathcal{H}(EF)$.   We shall show that $\mathcal{H}(E_{0})$ can be embedded into a larger Hilbert space $\mathcal{K}_{0}$, say by the embedding $T$, in such a way that the frame $\{ K_{E_{0}}(\lambda_{n}, \cdot) \}$ can be embedded into a Riesz basis for $\mathcal{K}_{0}$ having the form
\[ \{ \alpha_{n} T(K_{E_{0}}(\lambda_{n}, \cdot)) + \beta_{n} g_{n} \} \]
where $\alpha_{n}, \beta_{n} \in \mathbb{C}$ and $g_{n} \in \mathcal{K}_{0}$.

Unless specified otherwise, in this section $E,F \in \mathcal{HB}$, but need not be related to each other as in Theorem \ref{thm:main theorem of E and F}.

We define $\mathcal{I} : \mathcal{H}(E) \to \mathcal{H}(EF) : f \mapsto fF$; $\mathcal{I}$ is a linear isometry.
 \begin{lem}
The mapping $\mathcal{J} : \mathcal{H}(F) \to \mathcal{H}(EF)$ defined by $g \mapsto gE^{*}$ is a linear isometry.  Consequently, for every $g_{1}, g_{2} \in \mathcal{H}(F)$,
\begin{equation} \label{Eq:estar}
\langle g_{1} E^{*}, g_{2} E^{*} \rangle_{EF} = \langle g_{1}, g_{2} \rangle_{F}.
\end{equation}
\end{lem}

\begin{proof}
We claim that $\mathcal{J}$ is in fact well-defined:  for every $g \in \mathcal{H}(F)$, $gE \in \mathcal{H}(EF)$, so therefore $g^{*} E \in \mathcal{H}(EF)$.  Then, since $(g^{*} E)^{*} \in \mathcal{H}(EF)$, $g E^{*} \in \mathcal{H}(EF)$.  The linear and isometric conditions are easily verified.
\end{proof}


\begin{lem}
For $h \in \mathcal{H}(EF)$, $w \in \mathbb{C}$,
\[ \mathcal{I}^{*} h(w) = \int h(s) \dfrac{ \overline{ F(s) K_{E}(w,s) }}{ | E(s) |^2 | F(s) |^2 } ds. \]
Likewise,
\[ \mathcal{J}^{*} h(w) = \int h(s) \dfrac{ \overline{ E^{*}(s) K_{F}(w,s) }}{ | E(s) |^2 | F(s) |^2 } ds. \]
\end{lem}

\begin{proof}
Write $h = fF + h_{0}$ where $f \in \mathcal{H}(E)$ and $h_{0} \perp \mathcal{I}(\mathcal{H}(E))$.  Thus, we have
\begin{align*}
\int h(s) \dfrac{ \overline{ F(s) K_{E}(w,s) }}{ | E(s) |^2 | F(s) |^2 } ds &=  \langle h, F K_{E}(w, \cdot) \rangle_{EF} \\
&=  \langle f F, F K_{E}(w, \cdot) \rangle_{EF} \\
&=  \langle f ,  K_{E}(w, \cdot) \rangle_{E} \\
&= f(w).
\end{align*}
We have
\begin{align*}
\langle \mathcal{I}^{*} h , K_{E}(w, \cdot) \rangle_{E} &= \langle h , \mathcal{I} K_{E}(w, \cdot) \rangle_{EF} \\
&= \langle f F + h_{0} , F K_{E}(w, \cdot) \rangle_{EF} \\
&= \langle f ,  K_{E}(w, \cdot) \rangle_{E} \\
&= f(w).
\end{align*}

A similar calculation applied to $h = h_{1} + gE^{*}$ demonstrates the integral form of $\mathcal{J}^{*} h(w)$, utilizing Equation (\ref{Eq:estar}).
\end{proof}

Recall that the reproducing kernel in the space $\mathcal{H}(EF)$ is given by:
\begin{equation} \label{Eq:KEF}
K_{EF}(w,z) = \dfrac{\overline{E(w)} \overline{F(w)} E(z) F(z) - E(\overline{w}) F(\overline{w}) E^{*}(z) F^{*}(z) }{2 \pi i (\overline{w} - z) }
\end{equation}

\begin{lem}
For $w \in \mathbb{C}$,
\[ K_{EF}(w,z) = \overline{F(w)}F(z) K_{E}(w,z) + E(\overline{w}) E^{*}(z) K_{F}(w,z). \]
\end{lem}

\begin{proof}
Using Equation (\ref{eq:reproducingkerneldefn}), we first calculate:
\begin{equation} \label{Eq:Ke}
\overline{F(w)} F(z) K_{E}(w,z) = \dfrac{\overline{E(w)} \overline{F(w)} E(z) F(z) - E(\overline{w}) \overline{F(w)} E^{*}(z) F(z) }{2 \pi i (\overline{w} - z) }.
\end{equation}
We then calculate:
\begin{align*}
E(\overline{w}) E^{*}(z) K_{F}(w,z)  &= E(\overline{w}) E^{*}(z) \dfrac{ \overline{F(w)} F(z) - F(w) F^{*}(z)}{2 \pi i (\overline{w} - z)} \\
&= \dfrac{ E(\overline{w}) \overline{F(w)} E^{*}(z) F(z) - E(\overline{w}) F(w) E^{*}(z) F^{*}(z) }{2 \pi i (\overline{w} -z)}.
\end{align*}
Combining this with Equation (\ref{Eq:Ke}), we obtain the right hand side of Equation (\ref{Eq:KEF}).
\end{proof}

\begin{lem}
The following equation holds for the kernel $K_{EF}$:
\begin{equation} \label{Eq:KEF-IJ}
K_{EF}(w, z) = \overline{F(w)} [\mathcal{I}(K_{E}(w,\cdot))](z) + E(\overline{w}) [\mathcal{J}(K_{F}(w, \cdot))](z).
\end{equation}
\end{lem}

\begin{lem} \label{L:orthogonal}
The images of $\mathcal{I}$ and $\mathcal{J}$ are orthogonal in $\mathcal{H}(EF)$.
\end{lem}

\begin{proof}
For $f \in \mathcal{H}(E)$, $w \in \mathbb{C}$, we have
\begin{align*}
[\mathcal{I} (f)](w) &= f(w) F(w)  \\
&= \langle fF, K_{EF}(w, \cdot) \rangle_{EF} \\
&= \langle fF, \overline{F(w)} F K_{E}(w, \cdot) + E(\overline{w}) E^{*} K_{F}(w,\cdot) \rangle_{EF} \\
&= \langle fF, \overline{F(w)} F K_{E}(w, \cdot)  \rangle_{EF}  + \langle fF,  E(\overline{w}) E^{*} K_{F}(w,\cdot) \rangle_{EF} \\
&= F(w) \langle f, K_{E}(s, \cdot)\rangle_{E} + E^{*}(w) \langle fF,  E^{*} K_{F}(w,\cdot) \rangle_{EF} \\
&= f(w) F(w) + E^{*}(w) \langle fF,   \mathcal{J}( K_{F}(w,\cdot)) \rangle_{EF},
\end{align*}
from which it follows that $\mathcal{J}(K_{F}(w,\cdot))$ is orthogonal to $\mathcal{I}(\mathcal{H}(E))$ for any $w$ with $E^{*}(w) \neq 0$.  Since this collection has dense span in $\mathcal{J}(\mathcal{H}(F))$, the proof is complete.
\end{proof}

\begin{rem} \label{rem:H(EF) decomposition}
The images of $\mathcal{I}$ and $\mathcal{J}$ together span $\mathcal{H}(EF)$, as a consequence of Equation (\ref{Eq:KEF-IJ}).  Thus, for every $h \in \mathcal{H}(EF)$, $h = f F + g E^{*}$ for unique $f \in \mathcal{H}(E)$ and $g \in \mathcal{H}(F)$.  Let $P_{E}$ be the orthogonal projection of $\mathcal{H}(EF)$ onto the image of $\mathcal{I}$, and $P_{F}$ the projection onto the image of $\mathcal{J}$.  We have that $P_{E}(fF + gE^{*}) = fF$, and $P_{F} (fF + gE^{*}) = gE^{*}$.
\end{rem}

\begin{theorem} \label{T:PF1}
Suppose $\{ \lambda_{n} \} \subset \mathbb{R}$ is such that $\varphi_{EF}(\lambda_{n}) = n \pi + \alpha$ for some $\alpha \in [0,\pi)$ with the property that $\{ K_{EF}(\lambda_{n}, \cdot) \}$ is complete in $\mathcal{H}(EF)$.  Then
\[  \left\{ \dfrac{ \overline{F(\lambda_{n})} K_{E}(\lambda_{n}, \cdot)}{ \sqrt{K_{EF}( \lambda_{n}, \lambda_{n} )} } \right\}_{n \in \mathbb{Z}} \]
is a Parseval frame for $\mathcal{H}(E)$.  Consequently, for $f \in \mathcal{H}(E)$,
\begin{equation} \label{Eq:PFHE}
f(z) = \sum_{n} f(\lambda_{n}) \dfrac{ | F(\lambda_{n})|^2 K_{E}(\lambda_{n}, z) }{ K_{EF}( \lambda_{n}, \lambda_{n} ) }.
\end{equation}

Likewise,
\[  \left\{ \dfrac{ E(\lambda_{n}) K_{F}(\lambda_{n}, \cdot)}{ \sqrt{K_{EF}( \lambda_{n}, \lambda_{n} )} } \right\}_{n \in \mathbb{Z}} \]
is a Parseval frame for $\mathcal{H}(F)$, and for $g \in \mathcal{H}(F)$,
\begin{equation} \label{Eq:PFHF}
g(z) = \sum_{n} g(\lambda_{n}) \dfrac{ | E(\lambda_{n})|^2 K_{F}(\lambda_{n}, z) }{ K_{EF}( \lambda_{n}, \lambda_{n} ) }.
\end{equation}
\end{theorem}

\begin{proof}
We have by Equation (\ref{Eq:KEF-IJ}) and Lemma \ref{L:orthogonal} that
\[ P_{E} \left( \dfrac{ K_{EF}(\lambda_{n}, \cdot )}{ \sqrt{K_{EF}( \lambda_{n}, \lambda_{n} )} } \right) = \dfrac{ \overline{F(\lambda_{n})} \mathcal{I} (K_{E}(\lambda_{n}, \cdot) ) }{ \sqrt{K_{EF}( \lambda_{n}, \lambda_{n} )} } \]
Since our hypotheses imply that
\[  \left\{ \dfrac{ K_{EF}(\lambda_{n}, \cdot)}{ \sqrt{K_{EF}( \lambda_{n}, \lambda_{n} )} } \right\}_{n \in \mathbb{Z}} \]
is an orthonormal basis for $\mathcal{H}(EF)$, it follows that
\[ \left\{ \dfrac{ \overline{F(\lambda_{n})} \mathcal{I} (K_{E}(\lambda_{n}, \cdot) ) }{ \sqrt{K_{EF}( \lambda_{n}, \lambda_{n} )} } \right\}_{n \in \mathbb{Z}} \]
is a Parseval frame for $\mathcal{I}(\mathcal{H}(E))$.  Applying $\mathcal{I}^{*}$, which is an isometry from $\mathcal{I}(\mathcal{H}(E))$ to $\mathcal{H}(E)$, we obtain the first claim.

We now obtain
\begin{align*}
f(z) &= \sum_{n} \left\langle f , \dfrac{ \overline{F(\lambda_{n})} K_{E}(\lambda_{n}, \cdot)}{ \sqrt{K_{EF}( \lambda_{n}, \lambda_{n} )} } \right\rangle_{E} \dfrac{ \overline{F(\lambda_{n})} K_{E}(\lambda_{n}, z)}{ \sqrt{K_{EF}( \lambda_{n}, \lambda_{n} )} } \\
&= \sum_{n} f(\lambda_{n}) \dfrac{ | F(\lambda_{n})|^2 K_{E}(\lambda_{n}, z) }{ K_{EF}( \lambda_{n}, \lambda_{n} ) },
\end{align*}
as required.
\end{proof}

The following theorem shows that the Parseval frames for $\mathcal{H}(E)$ and $\mathcal{H}(F)$ given in Theorem \ref{T:PF1} are orthogonal.
\begin{theorem}
Suppose the hypotheses of Theorem \ref{T:PF1}. For every $f \in \mathcal{H}(E)$,
\begin{equation} \label{Eq:ortho1}
\sum_{n} f(\lambda_{n}) F(\lambda_{n}) \dfrac{ E(\lambda_{n}) K_{F}(\lambda_{n}, \cdot)}{ K_{EF}( \lambda_{n}, \lambda_{n} ) } = 0.
\end{equation}

Likewise, for every $g \in \mathcal{H}(F)$,
\begin{equation} \label{Eq:ortho2}
\sum_{n} g(\lambda_{n}) E^{*}(\lambda_{n}) \dfrac{  \overline{F(\lambda_{n})} K_{E}(\lambda_{n}, \cdot)}{ K_{EF}( \lambda_{n}, \lambda_{n} ) } = 0.
\end{equation}

\end{theorem}

\begin{proof}
For every $f \in \mathcal{H}(E)$, we have
\begin{align*}
f(z)F(z)  &= \sum_{n} \left\langle fF , \dfrac{K_{EF}(\lambda_{n}, \cdot)}{ \sqrt{K_{EF}(\lambda_{n}, \lambda_{n})}} \right\rangle_{EF} \dfrac{K_{EF}(\lambda_{n}, \cdot)}{ \sqrt{K_{EF}(\lambda_{n}, \lambda_{n})}} \\
&= \sum_{n} f(\lambda_{n}) F(\lambda_{n}) \dfrac{K_{EF}(\lambda_{n}, \cdot)}{ K_{EF}(\lambda_{n}, \lambda_{n})} \\
&= \sum_{n} f(\lambda_{n}) F(\lambda_{n})  \dfrac{\overline{F(\lambda_{n})} [\mathcal{I}(K_{E}(\lambda_{n},\cdot))](z) + E(\lambda_{n}) [\mathcal{J}(K_{F}(\lambda_{n}, \cdot))](z)}{K_{EF}(\lambda_{n}, \lambda_{n})}. \\
\end{align*}
Note that $\mathcal{J}^{*}(fF) = 0$, so applying $\mathcal{J}^{*}$ to the last line above, we obtain Equation (\ref{Eq:ortho1}).

An analogous argument applying $\mathcal{I}^{*}$ to $gE^{*}$ yields Equation (\ref{Eq:ortho2}).
\end{proof}

For convenience, let $h_{\alpha}(z) = e^{i \alpha} E(z)F(z) - e^{-i \alpha} E^{*}(z) F^{*}(z)$.

\begin{corollary}
Suppose that $\{ K_{E}(\lambda_{n}, \cdot ) \}_{n}$ is a Parseval frame for $\mathcal{H}(E)$ and $F \in \mathcal{HB}$ is such that $\varphi_{EF}(\lambda_{n}) = n \pi + \alpha$ for some $\alpha \in [0, \pi)$.  Then $\mathcal{H}(E)$ can be embedded into the Hilbert space $\mathcal{K}(\alpha)$ such that the Parseval frame is embedded into the orthonormal basis
\[  \left\{ \dfrac{\overline{F(\lambda_{n})} [\mathcal{I}(K_{E}(\lambda_{n},\cdot))](z)}{ \sqrt{K_{EF}(\lambda_{n}, \lambda_{n})}} \ \oplus  \ 
\dfrac{ E(\lambda_{n}) [\mathcal{J}(K_{F}(\lambda_{n}, \cdot))](z)  }{ \sqrt{K_{EF}(\lambda_{n}, \lambda_{n})} }
 \right\}_{n}\] 
where $\mathcal{K}(\alpha)$ is either $\mathcal{H}(EF)$ when $h_{\alpha} \notin \mathcal{H}(EF)$ or $\mathcal{H}(EF) \ominus span\{h_{\alpha} \}$ when $h_{\alpha} \in \mathcal{H}(EF)$.
\end{corollary}


\begin{corollary} \label{cor:naimark}
Assume the hypotheses of Theorem \ref{thm:main theorem of E and F}.  Then the sequence
\begin{equation} \label{Eq:RBS2}
\left\{ \dfrac{ \overline{F(\lambda_{n})} F(\cdot) K_{E_{0}}(\lambda_{n}, \cdot)}{\sqrt{K_{EF}(\lambda_{n}, \lambda_{n})}}
\oplus
\dfrac{ E(\lambda_{n}) E^{*}_{0}(\cdot) K_{F}(\lambda_{n}, \cdot)}{\sqrt{K_{EF}(\lambda_{n}, \lambda_{n})}}
\right\}
\end{equation}
is a Riesz basis in $\mathcal{K}_{0} \subset \mathcal{H}(E_{0}F)$.
\end{corollary}

\begin{proof}
By the proof of Theorem \ref{thm:similarity}, the sequence
\begin{equation} \label{Eq:E0-kern} 
\left\{ \dfrac{\overline{F(\lambda_{n})} K_{E_{0}}(\lambda_{n}, \cdot)}{\sqrt{K_{EF}(\lambda_{n}, \lambda_{n})}} \right\}
\end{equation}
is a frame in $\mathcal{H}(E)$, and is similar to the frame
\begin{equation} \label{Eq:E-kern}
\left\{ \dfrac{\overline{F(\lambda_{n})} K_{E}(\lambda_{n}, \cdot)}{\sqrt{K_{EF}(\lambda_{n}, \lambda_{n})}} \right\}
\end{equation}
Therefore, there exists an invertible operator $\mathcal{S} : \mathcal{H}(E) \to \mathcal{H}(E)$ that maps (\ref{Eq:E-kern}) to (\ref{Eq:E0-kern}).  Define a mapping $\mathfrak{J}: \mathcal{H}(EF) \to \mathcal{H}(E_{0}F)$ as follows:  for $f \in \mathcal{H}(E)$, $\mathfrak{J}(fF) = \mathcal{S}(f) F$, and for $g \in \mathcal{H}(F)$, $\mathfrak{J}(gE^{*}) = gE_{0}^{*}$.  This extends by linearity to $\mathcal{H}(EF)$.  By the orthogonality guaranteed by Lemma \ref{L:orthogonal},
\[ \| \mathfrak{J}(fF + gE^{*}) \|_{E_{0}F}^2 = \| \mathcal{S}(f) \|_{E_{0}}^2 + \| g \|_{F}^{2} \simeq \| \mathcal{S}(f) \|_{E}^{2} + \| g \|_{F}^{2} \simeq \| f \|_{E}^{2} + \| g \|_{F}^{2} = \| fF + gE^{*} \|_{EF}^2, \]
whence $\mathfrak{J}$ is continuous with closed range.  By Remark \ref{rem:H(EF) decomposition}, $\mathfrak{J}$ is onto, and hence is invertible.  Let $\mathcal{K}_{0} = \mathfrak{J}( \mathcal{K}(\pi/2))$.

From the proof of Theorem \ref{thm:main theorem of E and F}, the sequence $\{ \lambda_{n} \}$ coincides with the set $\{ t | \varphi_{EF}(t) = nt + \frac{\pi}{2} \}$, so the sequence
\[  \left\{ \dfrac{\overline{F(\lambda_{n})} [\mathcal{I}(K_{E}(\lambda_{n},\cdot))](z)}{ \sqrt{K_{EF}(\lambda_{n}, \lambda_{n})}} \ \oplus  \ 
\dfrac{ E(\lambda_{n}) [\mathcal{J}(K_{F}(\lambda_{n}, \cdot))](z)  }{ \sqrt{K_{EF}(\lambda_{n}, \lambda_{n})} }
 \right\}_{n}\] 
is an orthonormal basis of $\mathcal{K}(\pi/2)$ by Theorem \ref{thm:Orthonormal Basis in H(E) Theorem}.  We apply $\mathfrak{J}$ to this sequence to obtain
\[  \left\{ \dfrac{\overline{F(\lambda_{n})} F(\cdot) K_{E_{0}}(\lambda_{n},\cdot)}{ \sqrt{K_{EF}(\lambda_{n}, \lambda_{n})}} \ \oplus  \ 
\dfrac{ E(\lambda_{n}) E^{*}_{0}(\cdot) K_{F}(\lambda_{n}, \cdot)  }{ \sqrt{K_{EF}(\lambda_{n}, \lambda_{n})} }
 \right\}_{n}\] 
which is a Riesz basis for its span, which is $\mathcal{K}_{0}$.
\end{proof}

Thus, what we have here is $\mathcal{H}(E_{0})$ embedded into the larger space $\mathcal{K}_{0}$, and the frame $\{ K_{E_{0}}(\lambda_{n}, \cdot) \}_{n}$ embedded into the Riesz basis in (\ref{Eq:RBS2}).

\subsection{Multiplexing the Sampled Functions}

Multiplexing refers to the transmission of several signals simultaneously over a single communications channel.  Generically, multiplexing occurs when two (or more) signals $x$ and $y$ are encoded into $X$ and $Y$ in such a way that $x$ and $y$ can each be recovered from $X+Y$.  The  signals we consider here are elements of a de Branges space and the encoding involves the sampling of the signal.  Specifically, if $f \in \mathcal{H}(E)$ and $g \in \mathcal{H}(F)$, we encode both $f$ and $g$ into the \emph{multiplexed samples}:
\[ \{ f(\lambda_{n}) F(\lambda_{n}) + g(\lambda_{n}) E^{*}(\lambda_{n}) \}_{n} \]
which are transmitted in some fashion.  The goal then is to recover $f$ and $g$ from these mixed samples.

Consider the following toy example.  Suppose we have two bandlimited functions $f,g$ to transmit over a channel, where the band is $(-\pi, \pi)$.  We can modulate $f$ to obtain $\tilde{f}(x) = e^{- i \pi x} f(x)$ and $g$ to obtain $\tilde{g}(x) = e^{ i \pi x} g(x)$.  Then $\tilde{f}$ and $\tilde{g}$ are orthogonal in the space of bandlimited functions with the band $(-2\pi, 2\pi)$.  Therefore, we can encode $f$ and $g$ via the multiplexed samples
\[ \{ \tilde{f}(\frac{n}{2}) + \tilde{g}(\frac{n}{2}) \}_{n} \]
and recover $\tilde{f} + \tilde{g}$ from those samples.  Given $\tilde{f} + \tilde{g}$, we can project onto the subspace of bandlimited functions with band $(-2 \pi, 0)$ to recover $\tilde{f}$, and then unmodulate to obtain $f$.  Similarly, $g$ can be recovered from the multiplexed samples.

\begin{corollary} \label{C:multiplex}
Suppose the hypotheses of Theorem \ref{T:PF1}, and $f \in \mathcal{H}(E)$ and $g \in \mathcal{H}(F)$.  Given the samples $\{ f(\lambda_{n}) \}$ and $\{g(\lambda_{n})\}$, $f$ and $g$ can be reconstructed from the multiplexed samples as follows:
\begin{align}
f(z) &=  \sum_{n} \left( f(\lambda_{n}) F(\lambda_{n}) + g(\lambda_{n}) E^{*}(\lambda_{n}) \right) \dfrac{ \overline{F(\lambda_{n})} K_{E} (\lambda_{n},z) }{K_{EF}(\lambda_{n}, \lambda_{n} ) } \label{Eq:multiplex1} \\
g(z) &=  \sum_{n} \left( f(\lambda_{n}) F(\lambda_{n}) + g(\lambda_{n}) E^{*}(\lambda_{n}) \right) \dfrac{ E(\lambda_{n}) K_{F} (\lambda_{n},z) }{K_{EF}(\lambda_{n}, \lambda_{n} ) }.  \label{Eq:multiplex2}
\end{align}
\end{corollary}

\begin{proof}
Equations (\ref{Eq:multiplex1}) and (\ref{Eq:multiplex2}) follow immediately from Equations (\ref{Eq:PFHE} - \ref{Eq:ortho2}).
\end{proof}

\begin{remark}
We can apply Corollary \ref{C:multiplex} to our toy example as follows.  We let $E_{0}, E, F = e^{-i \pi z}$.  Since $EF = e^{- 2 i \pi z}$, we can sample functions in $\mathcal{H}(EF)$ at the half-integers (i.e. choose $\alpha = 0$), so for $f,g \in \mathcal{H}(e^{-i \pi z})$, the multiplexed samples are
\[ \{ e^{- i \pi \frac{n}{2} } f(\frac{n}{2}) + e^{ i \pi \frac{n}{2}} g(\frac{n}{2}) \} \]
which correspond exactly to the multiplexed samples of $\tilde{f} + \tilde{g}$.

Thus, we can view the embedding of $\mathcal{H}(E)$ into $\mathcal{H}(EF)$ as corresponding to a shift in the frequency band.
\end{remark}

\section{Sufficient Conditions}

The sufficiency of the Theorem \ref{thm:ocs} is more subtle, since we require a certain compatibility condition between $E$ and $F$.  The reason this is so is because of the lack of the Plancherel-Polya inequality in general for $\mathcal{H}(E)$.  See \cite{LS02a,aSW14a} for some discussion concerning the Plancherel-Polya inequality.  As we stated before, we can without loss of generality, restrict our attention to functions in the class $\mathcal{HB}$ which have no real roots.

\begin{theorem} \label{Th:sufficiency}
Suppose that $E_{0}, E, F \in \mathcal{HB}$ have no real roots such that $\mathcal{H}(E_{0}) = \mathcal{H}(E)$, and $\varphi_{F}^{\prime} \lesssim \varphi_{E}^{\prime}$ .  Suppose $\{ \lambda_{n} \}$ satisfies the equation $\varphi_{EF}(\lambda_{n}) = n \pi + \alpha$ for some $\alpha \in [0, \pi)$.  Then the sequence $\{ \lambda_{n} \}$ is a normalized sampling sequence for $\mathcal{H}(E_{0})$.
\end{theorem}

\begin{proof}
For $\alpha \in [0, \pi)$, let $\varphi_{EF}(\lambda_{n,\alpha}) = n \pi + \alpha$.  Suppose $\alpha$ is such that the sequence 
\begin{equation} \label{Eq:kernels6}
\left\{ \dfrac{ K_{EF}(\lambda_{n,\alpha}, \cdot)}{\sqrt{K_{EF}( \lambda_{n,\alpha}, \lambda_{n,\alpha}) }} \right\}
\end{equation}
is an orthonormal basis for $\mathcal{H}(EF)$.  Then by Theorem \ref{T:PF1}, the sequence
\[ \left\{ \dfrac{ \overline{F(\lambda_{n,\alpha})} K_{E}(\lambda_{n,\alpha}, \cdot)}{\sqrt{K_{EF}( \lambda_{n,\alpha}, \lambda_{n,\alpha}) }} \right\} \]
is a Parseval frame for $\mathcal{H}(E)$.  Therefore, we have for $f \in \mathcal{H}(E_{0})$:
\begin{align*}
\| f \|_{E_{0}}^{2} \simeq \| f\|_{E}^{2} & = \sum_{n} \left| \left\langle f , \dfrac{ \overline{F(\lambda_{n,\alpha})} K_{E}(\lambda_{n,\alpha}, \cdot)}{\sqrt{K_{EF}( \lambda_{n,\alpha}, \lambda_{n,\alpha}) }} \right\rangle \right|^2 \\
&= \sum_{n} \dfrac{ | F(\lambda_{n, \alpha}) |^2}{ K_{EF}( \lambda_{n,\alpha}, \lambda_{n,\alpha}) } | f(\lambda_{n, \alpha} )|^2 \\
&= \sum_{n} \dfrac{ \pi | f(\lambda_{n, \alpha} )|^2 }{ | E(\lambda_{n, \alpha}) |^2 (\varphi_{E}^{\prime} (\lambda_{n, \alpha}) + \varphi_{F}^{\prime} (\lambda_{n,\alpha}) ) } \\
&\leq \sum_{n} \dfrac{ | f(\lambda_{n, \alpha}) |^2 }{ K_{E}(\lambda_{n, \alpha}, \lambda_{n, \alpha} ) }.
\end{align*}
Therefore, the sequence satisfies the lower frame inequality (with the lower bound independent of the choice of $\alpha$).

Likewise, since $\varphi_{F}^{\prime} \lesssim \varphi_{E}^{\prime}$,
\begin{align*}
\| f \|_{E_{0}}^{2} & \simeq \sum_{n} \dfrac{ \pi | f(\lambda_{n, \alpha} )|^2 }{ | E(\lambda_{n, \alpha}) |^2 (\varphi_{E}^{\prime} (\lambda_{n, \alpha}) + \varphi_{F}^{\prime} (\lambda_{n,\alpha}) ) } \\
&\gtrsim \sum_{n} \dfrac{ \pi | f(\lambda_{n, \alpha} )|^2 }{ | E(\lambda_{n, \alpha}) |^2 (2 \varphi_{E}^{\prime} (\lambda_{n, \alpha})  ) } \\
&\simeq \sum_{n} \dfrac{ | f(\lambda_{n, \alpha}) |^2 }{ K_{E}(\lambda_{n, \alpha}, \lambda_{n, \alpha} ) }.
\end{align*}
Therefore, the sequence satisfies the upper frame inequality (with the upper bound independent of the choice of $\alpha$).

If $\alpha_{0}$ is such that the sequence in Equation (\ref{Eq:kernels6}) is incomplete in $\mathcal{H}(EF)$, then we can apply the Dominated Convergence Theorem to obtain
\[ 
\| f \|_{E_{0}}^2 \simeq \lim_{\alpha \to \alpha_{0}} \sum_{n} \dfrac{ | f(\lambda_{n, \alpha}) |^2 }{ K_{E}(\lambda_{n, \alpha}, \lambda_{n, \alpha} ) } = \sum_{n} \dfrac{ | f(\lambda_{n, \alpha_{0}}) |^2 }{ K_{E}(\lambda_{n, \alpha_{0}}, \lambda_{n, \alpha_{0}} ) }.
\]
\end{proof}

\begin{corollary}
Assume the hypotheses of Theorem \ref{Th:sufficiency}.  If $\{ \lambda_{n} \}$ is the zero set of $EF + E^{*}F^{*}$, then $\{ \lambda_{n} \}$ is a normalized sampling set for $\mathcal{H}(E_{0})$.
\end{corollary}

\begin{proof}
The zero set for $EF + E^{*} F^{*}$ coincides with $\{ \lambda_{n, \frac{\pi}{2}} \}$.
\end{proof}

\begin{corollary}
Assume the hypotheses of Theorem \ref{Th:sufficiency}.  Assume also that $K_{E_{0}}(x, x) \simeq 1$.  Then the zero set of $EF + E^{*}F^{*}$ is a (non-normalized) sampling sequence for $\mathcal{H}(E_{0})$.
\end{corollary}

\begin{proof}
This follows from Lemma \ref{lem:equivalent kernel functions} below.  Indeed, if $K_{E}(x, x) \simeq 1$, we have 
\[ \| f \|_{E_{0}}^{2} \simeq \sum_{n} \dfrac{ | f(\lambda_{n, \alpha}) |^2 }{ K_{E}(\lambda_{n, \alpha}, \lambda_{n, \alpha} ) } \simeq \sum_{n} | f(\lambda_{n, \alpha}) |^2. \]
\end{proof}

 \begin{lem}\label{lem:equivalent kernel functions}
For $E_{1}, E_{2} \in \mathcal{HB}$, if $H(E_{1})\simeq H(E_{2})$, then $K_{1}(x,x)\simeq K_{2}(x,x)$ for all $x\in \mathbb{R}$, where $K_{1},K_{2}$ are the reproducing kernels of $H(E_{1})$ and $H(E_{2})$, respectively.
 \end{lem}
\begin{proof} First note that by the reproducing kernel property (\ref{eq:reproducingkernelProperty}), $\|K_{1}(x,.)\|_{H(E_{1})}^{2}=K_{1}(x,x)$, and $\|K_{2}(x,.)\|_{H(E_{2})}^{2}=K_{2}(x,x)$  for $x\in \mathbb{R}$. On the other hand,
   \begin{eqnarray*}
    \|K_{1}\|_{H(E_{1})} 
                        &=& \sup\limits_{\substack {f\in H(E_{1})\\\|f\|_{H(E_{1})}=1}}|f(z)|\\
                        &\simeq& \sup\limits_{\substack {f\in H(E_{2})\\\|f\|_{H(E_{2})}=1}} |f(z)|\\
                        &=&\|K_{2}\|_{H(E_{2})}
\end{eqnarray*}
hence, $K_{1}(x,x)\simeq K_{2}(x,x)$ for all $x\in \mathbb{R}$.
 \end{proof}


\begin{remark}
We note that the theorem proven by Ortega-Cerda and Seip for the Paley-Wiener space automatically satisfies the condition $K_{E_{0}}(x, x) \simeq 1$.  Also, they do not require the condition $\varphi_{F}^{\prime} \lesssim \varphi_{E}^{\prime}$, again because the Plancherel-Polya inequality holds in the Paley-Wiener space, and so the upper frame bound is satisfied (they assume a priori that the sequences are separated).
\end{remark}


\providecommand{\bysame}{\leavevmode\hbox to3em{\hrulefill}\thinspace}
\providecommand{\MR}{\relax\ifhmode\unskip\space\fi MR }
\providecommand{\MRhref}[2]{%
  \href{http://www.ams.org/mathscinet-getitem?mr=#1}{#2}
}
\providecommand{\href}[2]{#2}

\end{document}